\documentclass[11pt, a4paper]{article}

\usepackage[T1]{fontenc}
\usepackage[utf8]{inputenc} 
\usepackage{graphicx,color} 
\usepackage{epsfig} 
\usepackage{times} 
\usepackage{amsmath} 
\usepackage{amssymb}  
\usepackage{caption}
\usepackage{subcaption}
\usepackage{stackrel}

\usepackage{url}
\usepackage{multirow}
\usepackage{cite}

\usepackage{algorithm}
\usepackage{algpseudocode}

\makeatletter
\renewcommand{\Statex}{\item[]\hskip\ALG@thistlm}
\makeatother

\usepackage{pgf}
\usepackage{pgfplots}
\usepgfplotslibrary{fillbetween}
\usepackage{tikz}
\usetikzlibrary{arrows,automata}
\usetikzlibrary{intersections}
\usetikzlibrary{calc}
\usetikzlibrary{shapes}
\usetikzlibrary{positioning}
\usetikzlibrary{decorations.text}
\pgfdeclarelayer{background}
\pgfdeclarelayer{foreground}
\pgfsetlayers{background,main,foreground}
\usetikzlibrary{calc,patterns,decorations.pathmorphing,decorations.markings}

\definecolor{mycol}{RGB}{19,48,128}
\usepackage{hyperref} 
\hypersetup{
	colorlinks = true,
	citecolor = mycol,
	linkcolor = mycol,
	urlcolor = black
}

\newtheorem{theorem}{Theorem}[section]
\newtheorem{lemma}[theorem]{Lemma}

\newtheorem{assumption}[theorem]{Assumption}

\newtheorem{remark}[theorem]{Remark}

\newcommand{\R}{\mathbb{R}}

\newcommand{\bG}{G}
\newcommand{\bGs}{G_s}

\newcommand{\bQ}{Q}

\newcommand{\bW}{W}
\newcommand{\xstar}{x^\star}
\newcommand{\lstar}{\lambda^\star}
\newcommand{\nstar}{\nu^\star}
\newcommand{\T}{\top}

\newcommand{\calX}{\mathcal{X}}

\newcommand{\proof}{\noindent{\em Proof}.~}
\newcommand{\QED}{$\Box$}
\newcommand{\cvd}{\hfill\QED}

\algrenewcommand\algorithmicindent{1.2em}
\algrenewcommand\algorithmiccomment[1]{\hfill{\footnotesize[\textit{\textsf{#1}}]}}

\hypersetup{
  colorlinks=true,
  linkcolor=blue!60!black,
  citecolor=green!50!black,
  urlcolor=blue!70!black,
}

\begin{document}

\title{Towards A Goldfarb-Idnani Variant for Strongly Monotone Linear-Quadratic Games}

\author{Alberto Bemporad
\thanks{The author is with the IMT School for Advanced Studies Lucca, Italy (email: {\tt \footnotesize alberto.bemporad@imtlucca.it}).}%
}

\maketitle

\begin{abstract}
We analyze a simple variant of the Goldfarb-Idnani (GI) dual active-set method for computing variational generalized Nash equilibria of strongly monotone $N$-player games with convex quadratic costs and shared affine inequality and equality constraints. We show that several properties of the GI algorithm are maintained in spite of having a possibly non-symmetric pseudogradient matrix in the joint KKT system of the game, although convergence to an existing equilibrium is not guaranteed as in the original algorithm. Our numerical results show that the method is potentially competitive with alternative state-of-the-art algorithms, including for computing solutions of game-theoretic linear model predictive control laws.
\end{abstract}

\section{Introduction}
\label{sec:intro}
Linear-quadratic generalized Nash equilibrium problems (LQ-GNEPs)
arise when multiple decision makers optimize convex quadratic objectives under
local and shared affine constraints. This structure is central in
game-theoretic model predictive control (GT-MPC) of linear systems, where finite-horizon dynamic games must be solved repeatedly and reliably online in fast-sampling applications, such as unmanned aerial vehicles~\cite{DPM23}, autonomous driving~\cite{ABB26}, and demand-side energy management~\cite{HBLD22}.

In the presence of shared constraints, a common approach is to seek a 
variational generalized Nash equilibrium (v-GNE), in
which all players assign the same multiplier to the common constraints~\cite{rosen1965}.
Indeed, most of the existing algorithms for solving GNEPs are designed to compute v-GNEs, as
conditions of existence and uniqueness of solutions can be well characterized~\cite{FP03}.

Several numerical methods exist for solving LQ-GNEPs. Since, under inequality constraints, their variational equilibria can be characterized as  solutions of an affine variational inequality (AVI), or equivalently as a linear complementarity problem (LCP) with a possibly non-symmetric LCP matrix, pivotal methods have been proposed for AVI/LCP reformulations~\cite{CF96,SPS13,KHF18}. Interior-point methods are another important class of algorithms for solving classes of GNEPs~\cite{DFKS11}, and LQ-GNEPs in particular~\cite{LL25}. 

Splitting-based operator frameworks were also proposed to solve v-GNEs~\cite{YP19,BYGP22,BK21}, see also the recent package~\cite{MBCEDG25} collecting several of them; such methods typically have inexpensive iterations and favorable communication structure, but their convergence is asymptotic, which can be restrictive for embedded GT-MPC where highly accurate solutions are usually required at each sampling instant within short time intervals. The hybrid active-set method DR-DAQP was recently proposed in~\cite{ABB26} to effectively combine the robustness of Douglas-Rachford iterations with the finite termination properties of active-set methods. 

Augmented Lagrangian formulations were proposed in~\cite{KS16,
LSM22} and distributed algorithms in~\cite{YP19,TSN21,BYGP22,TK18}. A learning-based approach relying on best-response measurements was proposed in~\cite{FB25}. For parametric (possibly non-variational) LQ-GNEPs of moderate size, multiparametric programming was proposed recently in~\cite{HB26} to explicitly characterize the (possibly multiple) solutions as a function of the parameters. 

LQ-GNEPs can also be solved by reformulating the KKT conditions as a mixed-integer linear or quadratic program (MILP/MIQP)~\cite{Bem26},
which, although often numerically inefficient, enables finding non-variational GNEs and extracting
multiple equilibria. General-purpose numerical approaches that minimize the residuals of the KKT conditions~\cite{DFKS11,Bem26} can also be used, in principle,
to seek (variational) equilibria of LQ-GNEPs, although they are not guaranteed in general to converge to a GNE and are rather inefficient due to the use of (nonconvex) nonlinear programming. 

\subsection{Contribution}
Motivated by real-time GT-MPC applications, in this paper we analyze how to modify the Goldfarb-Idnani (GI) dual active-set
algorithm~\cite{GI83}, designed for symmetric positive-definite quadratic programming (QP) problems, to the non-symmetric, strongly monotone affine pseudogradient appearing in the VI/KKT characterization of a v-GNE. The resulting method, which we call GoldNash, enjoys high-accuracy and low-cost working-set update properties of GI, although it is not per se guaranteed to always find a solution if one exists. An experimental Python implementation of the method is available in the \texttt{NashOpt} library~\cite{Bem26}.

The rest of the paper is organized as follows. In Section~\ref{sec:problem} we formulate the problem of computing v-GNEs of LQ-GNEPs and in Section~\ref{sec:prelim} we analyze relevant properties of the associated pseudogradient matrix. In Section~\ref{sec:algorithm} we present our algorithm and in Section~\ref{sec:results} we report numerical results comparing it with alternative state-of-the-art algorithms for solving LQ-GNEPs, including those arising in GT-MPC problems. Finally, we draw some conclusions in Section~\ref{sec:conclusion}.

\section{Problem Formulation}
\label{sec:problem}

Consider an $N$-player non-cooperative generalized Nash equilibrium problem with 
decision vectors $x_i \in \R^{n_i}$, $i = 1, \ldots, N$, and denote
by $x = (x_1, \ldots, x_N) \in \R^n$, $n = \sum_{i=1}^N n_i$,
the joint strategy profile. Each player $i$ solves the following optimization problem:
\begin{equation}
    \begin{aligned}
    \min_{x_i} & \tfrac{1}{2} x^\T \bQ^{(i)} x + (c^{(i)})^\T x\\
    \textrm{s.t. }& Ax \leq b,\ Ex = f
    \end{aligned}
    \label{eq:QPi}
\end{equation}
where $\bQ^{(i)}\in \R^{n \times n}$, $c^{(i)} \in \R^n$, $A \in \R^{m \times n}$, $b \in \R^m$, $E \in \R^{q \times n}$, $f \in \R^q$ ($m,q\geq 0$).
Without loss of generality, we assume $\bQ^{(i)}=(\bQ^{(i)})^\T$. We also assume that each QP in~\eqref{eq:QPi} is strictly convex, i.e., $\bQ^{(i)}_{ii} \succ 0$, $\forall i = 1, \ldots, N$,
where $\bQ^{(i)}_{ii}$ denotes the $i$-th diagonal block of $\bQ^{(i)}$.
We consider the constraint set $\calX=\{x \in \R^n : Ax \leq b,\; Ex = f\}$ as \emph{shared}, i.e., it is common to all players. Possible local inequality constraints only involving $x_i$, such as lower and upper bounds, are assumed to be incorporated into $Ax \leq b$, and local equality constraints into $Ex = f$.

By stacking the gradients $\bQ^{(i)}_{i, :}x + c^{(i)}_{i}$ of each agent's cost with respect to its own decision variable, we can define the \emph{pseudogradient} $F: \R^n \to \R^n$ of the game 
\[
  F(x) = \bG x + g
\]
where $\bG \in \R^{n \times n}$, $g \in \R^n$, and
\begin{equation}\label{eq:Gg_assembly}
  \bG_{i, :} = \bQ^{(i)}_{i,:}, \qquad
  g_{i}   = c^{(i)}_{i}, \qquad i = 1,\ldots,N.
\end{equation}

A vector $\xstar \in \calX$ is a \emph{generalized Nash equilibrium} (GNE) if each vector $\xstar_i$ solves~\eqref{eq:QPi} for $x_{-i} = \xstar_{-i}$, for all $i = 1, \ldots, N$, where $x_{-i}$ denotes all the components of $x$ except those in $x_i$. Under the
additional condition that the vector $\lstar \in \R^m$ of Lagrange multipliers associated with the shared constraints is the same for all players, such an equilibrium is called a \emph{variational GNE} (v-GNE) and can be characterized by the joint Karush-Kuhn-Tucker (KKT) conditions~\cite[Theorem 12.1]{NW06} of the $N$ optimization problems~\eqref{eq:QPi}: 
\begin{subequations}
\begin{align}
  \bG \xstar + g + A^\T \lstar + E^\T \nstar &= 0 \label{eq:KKT_stat}\\
  E \xstar &= f                \label{eq:KKT_eq}\\
  A \xstar &\leq b              \label{eq:KKT_prim}\\
  \lstar   &\ge 0              \label{eq:KKT_dual}\\
  (\lstar)^\T (A \xstar - b)   &= 0 \label{eq:KKT_comp}
\end{align}
\label{eq:KKT}%
\end{subequations}
where $\nstar \in \R^q$ is the shared dual vector for the equality constraints and is unrestricted in sign.
Note that local constraints only involving $x_i$ do not contribute to the optimality conditions of each other player $j\neq i$, as the corresponding rows of $A$ (resp.\ $E$) are zero in the columns corresponding to $x_j$. 

\section{Preliminaries}
\label{sec:prelim}
\subsection{Reformulations as Linear Complementarity Problems}
\label{sec:lcp_reformulation}
The KKT conditions~\eqref{eq:KKT} can be reformulated as linear complementarity problems (LCPs) and solved by the classical pivotal algorithm of Lemke~\cite{Lem65,CPS09}. We distinguish two variants of the reformulation, which we refer to as the \emph{primal} and \emph{dual} Lemke approaches, described below.
\subsubsection{Primal Lemke} 
When no equality constraints are present and $x$ has finite lower bounds, $x\geq \ell$, by setting $y=x-\ell$, $s=b-Ay-A\ell$, and calling $\mu$ the Lagrange multipliers of the lower-bound constraints, the KKT conditions~\eqref{eq:KKT} can be rearranged to give the following linear complementarity problem (LCP):
\[
    0 \;\leq\;
    \begin{bmatrix} y \\ \lambda \end{bmatrix}
    \;\perp\;
    \underbrace{\begin{bmatrix} \bG & A^\T \\ -A & 0 \end{bmatrix}}_{M}
    \begin{bmatrix} y \\ \lambda \end{bmatrix}
    +
    \underbrace{\begin{bmatrix} \bG\ell + g \\ b - A\ell \end{bmatrix}}_{q}
    =
    \begin{bmatrix} \mu \\ s \end{bmatrix}
    \;\ge\; 0.
\]
This LCP can be solved by Lemke's algorithm~\cite{Lem65,CPS09}.
\subsubsection{Dual Lemke}
\label{sec:lcp_reformulation-dual}
In the absence of equality constraints, since $\bG$ is invertible
by Lemma~\ref{lem:Ginv_pd}, we can replace $x= -\bG^{-1}(g + A^\T \lambda)$ in the KKT conditions~\eqref{eq:KKT} and get the following LCP
\[
    \begin{aligned}
    0\leq \lambda \perp \underbrace{A\bG^{-1} A^\T}_{M}\lambda + \underbrace{b + A\bG^{-1}g}_{q} \geq 0.
    \end{aligned}
\]
In the presence of equality constraints, the same approach can be applied by first eliminating both $x$ and $\nu$ from~\eqref{eq:KKT_stat}--\eqref{eq:KKT_eq}
\begin{equation}
    \begin{bmatrix}x\\\nu\end{bmatrix} = -\begin{bmatrix} \bG & E^\T \\ E & 0 \end{bmatrix}^{-1} \begin{bmatrix} g + A^\T \lambda \\ f \end{bmatrix}
\label{eq:dual_lemke_LCP}
\end{equation}
and then substituting the expression for $x$ as a function of $\lambda$ into the primal-feasibility~\eqref{eq:KKT_prim} and complementarity conditions~\eqref{eq:KKT_comp} to get another LCP involving only $\lambda$.
\subsection{Properties of the pseudogradient matrix}
In general, the pseudogradient matrix $\bG$ is not symmetric, which makes the KKT conditions above not equivalent to those of a standard QP. We denote by $\bGs = \tfrac{1}{2}(\bG + \bG^\T)$ the \emph{symmetric part} of $\bG$. In the special case $\bG = \bGs$, the game is a potential game, whose variational GNEs are indeed the solutions of a standard QP. 

To adapt the GI algorithm~\cite{GI83} to attempt solving the v-GNE problem we posed, we need to verify certain properties of the pseudogradient matrix $\bG$, due to its possible asymmetry. In particular, we want to show that the add/drop logic of~\cite{GI83} does not depend on the symmetry of the pseudogradient matrix.

\begin{assumption}[Strong monotonicity]\label{ass:strong_mono}
The symmetric part of the pseudogradient matrix satisfies $\bGs=\bGs^\T \succ 0$.
\end{assumption}

\begin{assumption}
\label{ass:calX}
The shared constraint set $\calX = \{x \in \R^n : Ax \leq b,\; Ex = f\}$ is non-empty.
\end{assumption}
Under Assumptions~\ref{ass:calX} and~\ref{ass:strong_mono}, a v-GNE solution
$\xstar$ exists and is unique~\cite[Thm.~2.3.3]{FP03}.
Without loss of generality, we also assume $\operatorname{rank}(E) = q$. In fact, in the case $E$ has linearly dependent rows, one can perform a QR decomposition with pivoting $\Pi [E\ -f] = Q \left[\begin{smallmatrix} R_{11} & R_{12} \\ 0 & R_{22} \end{smallmatrix}\right]$, where $R_{22}=0$ is guaranteed as $\{x:\ Ex=f\} \neq \emptyset$ by Assumption~\ref{ass:calX}, and replace
the equality constraints by $[R_{11}\ R_{12}]\left[\begin{smallmatrix} x' \\ 1 \end{smallmatrix}\right]=0$.

\begin{lemma}
Under Assumption~\ref{ass:strong_mono}, the inverse $\bG^{-1}$ of the pseudogradient matrix $\bG$ exists and $x^\T\bG^{-1}x>0$ for all $x\neq0$.
\label{lem:Ginv_pd}
\end{lemma}

\proof
Suppose by contradiction that $\bG$ is not full rank, i.e., $\bG x = 0$ for some $x \neq 0$.
Then $0=x^\T \bG x = x^\T \bGs x$, contradicting $\bGs \succ 0$, and hence $\bG^{-1}$ exists.
Moreover, consider any vector $x\neq 0$ and let $y = \bG^{-1}x$. Then,
$x^\T \bG^{-1} x = (\bG y)^\T \bG^{-1} (\bG y) = y^\T \bG^\T y = y^\T\bGs y>0$, as $y\neq 0$ and 
$\bGs \succ 0$.
\cvd

\section{GI Algorithm}
\label{sec:algorithm}

\begin{algorithm}[h!]
\caption{GoldNash: A v-GNE variant of the dual active-set solver by Goldfarb-Idnani~\cite{GI83}}
\label{alg:goldnash}
\begin{algorithmic}[1]
\Require $\{\bQ^{(i)},c^{(i)}\}_{i=1}^N$;\; $A\in\R^{m\times n}$, $b\in\R^m$;\;
         $E\in\R^{q\times n}$, $f\in\R^q$;\; $\varepsilon>0$; $M>0$.
\Ensure $\xstar\in\R^n$, $\lstar\in\R^m_{\ge 0}$, $\nstar\in\R^q$, \texttt{status}.
\State Assemble $\bG$, $g$ via~\eqref{eq:Gg_assembly}; compute LU factorization $P\bG=LU$\label{line:lu}.
\State Precompute $Y_A\leftarrow\bG^{-1}A^\T$ and $Y_E\leftarrow\bG^{-1}E^\T$.\label{line:lu2}
\State $x\leftarrow{-\bG^{-1}g}$;\;
       $\nu\leftarrow(EY_E)^{-1}(Ex-f)$;\;
       $x\leftarrow x-Y_E\nu$;\;
       $\lambda\leftarrow 0_m$;\;
       $\bW\leftarrow\emptyset$;\;
       $j\leftarrow 0$.\label{line:init}
\Repeat \Comment{Outer loop}
  \State $p\leftarrow\arg\max_k(Ax-b)_k$;\; $\rho_p\leftarrow(Ax-b)_p$.
  \If{$\rho_p\le\varepsilon$} \Return $(x,\lambda,\nu,\texttt{optimal})$. \EndIf
  \State $\mu_p\leftarrow 0$;\; $j\leftarrow j+1$.
  \Repeat \Comment{Inner loop: target constraint $p$}
    \State $y_p\leftarrow[Y_A]_{:,p}$;\;
           $Y_{\bar\bW}\leftarrow\bigl[[Y_A]_{:,\bW},\,Y_E\bigr]$;\;
           $\bar{A}_\bW\leftarrow\bigl[\begin{smallmatrix}A_\bW\\E\end{smallmatrix}\bigr]$.
    \State $\bar{S}_\bW\leftarrow\bar{A}_\bW Y_{\bar\bW}$;\;
           $\bar{r}_\bW=(r_\bW,s)\leftarrow\bar{S}_\bW^{-1}(\bar{A}_\bW\,y_p)$;\label{line:dual_step}
    \Statex $z\leftarrow Y_{\bar\bW}\bar{r}_\bW-y_p$;\; $j\leftarrow j+1$.
    \If{$a_p^\T z<-\varepsilon$}\label{line:t2-if}
      \State $t_2\leftarrow-(a_p^\T x-b_p)\,/\,(a_p^\T z)$.
    \Else
      \State $t_2\leftarrow+\infty$.
    \EndIf\label{line:t2-end}
    \State $t_1\leftarrow\min_{k\in\bW,\,r_k>0}\lambda_k/r_k$
           \quad ($+\infty$ if $r_k\leq 0$ $\forall k\in\bW$).
    \If{$t_1=t_2=+\infty$} \Return $(\cdot,\cdot,\cdot,\texttt{infeasible})$. \EndIf
    \State $t\leftarrow\min(t_1,t_2)$;\;
           $x\leftarrow x+tz$;\;
           $\lambda_\bW\leftarrow\lambda_\bW-tr_\bW$;\;
    \Statex $\nu\leftarrow\nu-ts$;\;
           $\mu_p\leftarrow\mu_p+t$;\;
           $\lambda_p\leftarrow\mu_p$.
    \If{$t_2\leq t_1$}
      \State $\bW\leftarrow\bW\cup\{p\}$;\; \textbf{break}.\label{line:add} \Comment{Add $p$ to $\bW$}
    \Else
      \State $j^\star\leftarrow\arg\min_{k\in\bW,\,r_k>0}\lambda_k/r_k$;\;
      \Statex $\lambda_{j^\star}\leftarrow 0$;\;
             $\bW\leftarrow\bW\setminus\{j^\star\}$. \Comment{Drop $j^\star$ from $\bW$}
    \EndIf
  \Until{inner loop exits or $j\geq M$}
\Until{outer loop exits or $j\geq M$}
\State \Return $(\cdot,\cdot,\cdot,\texttt{unsolved})$.
\end{algorithmic}
\end{algorithm}

Our adaptation of the GI algorithm~\cite{GI83} to attempt solving the v-GNE problem
is summarized in Algorithm~\ref{alg:goldnash} for the convenience of the reader. The algorithm maintains a working set $\bW \subseteq \{1, \ldots, m\}$ of
indices of \emph{inequality} constraints treated as active at the equilibrium, 
and a \emph{tentative} multiplier
$\mu_p$ for the currently targeted violated inequality constraint $p \notin \bW$.
Equality constraints are always active and kept separate. The current iterate
$(x, \nu, \lambda_\bW, \mu_p)$ satisfies the stationarity invariance condition
\begin{subequations}
\begin{equation}\label{eq:invariant}
  \bG x + g + E^\T \nu + \sum_{k \in \bW} \lambda_k a_k + \mu_p\, a_p = 0
\end{equation}
as required by~\eqref{eq:KKT_stat}, where $a_k^\T$ denotes the $k$-th row of $A$,
and
\begin{eqnarray}
 A_\bW x &=& b_\bW \label{eq:Aw-tight} \\
 E x &=& f. \label{eq:E-tight}
\end{eqnarray}
\label{eq:invariant-conditions}%
\end{subequations}
The multipliers $\lambda_k$ for $k \in \bW$ are restricted to stay nonnegative, $\lambda_k \geq 0$.
The algorithm is allowed to run for at most $M$ iterations.

The algorithm starts by solving the equality-constrained Nash equilibrium with respect to $x$ and $\nu$:
\begin{equation}\label{eq:init}
  \begin{bmatrix} \bG & E^\T \\ E & 0 \end{bmatrix}
  \begin{bmatrix} x \\ \nu \end{bmatrix}
  = \begin{bmatrix} -g \\ f \end{bmatrix}
\end{equation}
by block elimination (Step~\ref{line:init}), i.e., by substituting
$x = -G^{-1}(g + E^\top\nu)$ in $Ex = f$ to evaluate $\nu$ from
$(EG^{-1}E^\top)\nu = -EG^{-1}g-f$ and then recovering $x$.
Note that~\eqref{eq:init} has a unique solution by Lemma~\ref{lem:Ginv_pd} and the assumed full row-rank of $E$.
The initialization sets $\bW = \emptyset$ and $\mu_p = 0$, so that~\eqref{eq:invariant} is satisfied.

Consider now a generic step of the algorithm, where the current working set $\bW$ is fixed and a new violated inequality constraint $p \notin \bW$ is targeted for addition. The algorithm performs the following updates:
\begin{equation}
\begin{aligned}
x &\leftarrow x + tz \\
\lambda_\bW &\leftarrow \lambda_\bW - t\,r_\bW \\
\nu &\leftarrow \nu - t\,s \\
\mu_p &\leftarrow \mu_p + t
\end{aligned}
\label{eq:update}
\end{equation}
for a properly chosen primal direction $z \in \R^n$, dual direction $r_\bW \in \R^{|\bW|}$ for the active inequality multipliers, dual direction $s \in \R^q$ for the equality multipliers, and positive step-size 
$t\in\R_{>0}\cup\{+\infty\}$. These are determined by the inner loop such that the conditions~\eqref{eq:invariant-conditions} are preserved after the update:
\begin{subequations}
\begin{equation}
  \bG (x+tz) + g + E^\T\!(\nu - ts) +\! \sum_{k \in \bW} (\lambda_k - t r_k) a_k + (\mu_p+t) a_p = 0
\label{eq:stationarity-after}
\end{equation}
\begin{eqnarray}
\nonumber\\[-2em]
 A_\bW(x+tz) &=& b_\bW \label{eq:Aw-tight-t} \\
 E(x+tz) &=& f. \label{eq:E-tight-t}
\end{eqnarray}
\end{subequations}
Since $t>0$ and, before the update, $x$ satisfies~\eqref{eq:Aw-tight-t}--\eqref{eq:E-tight-t},
these imply that
\begin{equation}
  \bar{A}_\bW z = 0,\quad \bar{A}_\bW = \begin{bmatrix} A_\bW \\ E \end{bmatrix}.
  \label{eq:Aw-tight-z}
\end{equation}
By using~\eqref{eq:invariant}, the condition in~\eqref{eq:stationarity-after} simplifies to
$t(\bG z - \bar{A}_\bW^\T \bar{r}_\bW + a_p) = 0$ or equivalently, since $t>0$, to
\begin{subequations}\label{eq:step_system}
\begin{equation}
  z = \bG^{-1}(\bar{A}_\bW^\T \bar{r}_\bW - a_p).
\label{eq:step_G}
\end{equation}
Substituting~\eqref{eq:step_G} into~\eqref{eq:Aw-tight-z} gives
$\bar{A}_\bW \bG^{-1}(\bar{A}_\bW^\T \bar{r}_\bW - a_p) = 0$,
and finally the dual direction
\begin{equation}\label{eq:dual_step}
  \bar{r}_\bW = \bar{S}_\bW^{-1}(\bar{A}_\bW \bG^{-1} a_p)
\end{equation}
\end{subequations}
where $\bar{S}_\bW = \bar{A}_\bW\,\bG^{-1}\bar{A}_\bW^\T$ is invertible, as shown
by the following lemma.

\begin{lemma}\label{lem:schur}
Let $\bW \subseteq \{1, \ldots, m\}$ be any index set,
let $A_\bW \in \R^{|\bW| \times n}$ be the rows of $A$ indexed by $\bW$,
and let $\bar{A}_\bW$ be defined as in~\eqref{eq:Aw-tight-z}.
If $\bar{A}_\bW$ has full row rank and Assumption~\ref{ass:strong_mono} holds, then
\begin{equation}\label{eq:schur}
  \bar{S}_\bW = \bar{A}_\bW\, \bG^{-1} \bar{A}_\bW^\T \;\in\; \R^{(|\bW|+q) \times (|\bW|+q)}
\end{equation}
is invertible.
\end{lemma}

\proof
Suppose $\bar{S}_\bW v = 0$ for some $v \in \R^{|\bW|+q}$, $v\neq 0$, and let $z=\bar{A}_\bW^\T v$.
Clearly, $z\neq 0$ since the rows of $\bar{A}_\bW$ are linearly independent.
Then $0 = v^\T \bar{S}_\bW v = (\bar{A}_\bW^\T v)^\T \bG^{-1} (\bar{A}_\bW^\T v)=z^\T \bG^{-1} z > 0$
by Lemma~\ref{lem:Ginv_pd}, a contradiction. Therefore, $v = 0$ and $\ker(\bar{S}_\bW) = \{0\}$, so $\bar{S}_\bW$ is invertible.
\cvd

Note that
$\bar{S}_\bW$ is not symmetric in general, since $\bar{S}_\bW^\T = \bar{A}_\bW\bG^{-\T}\bar{A}_\bW^\T
$ and $\bG$ may not be symmetric.
In the original GI algorithm, $G=G^\top\succ0$ implies
$G^{-1}=G^{-\top}\succ0$, hence $S_\bW=A_\bW G^{-1}A_\bW^\top$
is symmetric positive definite whenever $A_\bW$ has full row rank.

The second symmetry-dependent fact in the GI proof is the property
$a_p^\T z<0$ for each nonzero primal step $z$.  The following lemma shows that
the same property holds in the v-GNE setting.

\begin{lemma}\label{lem:descent}
Let $\bar{A}_\bW$ be full row rank, let $p\notin\bW$, assume $a_p\neq0$
and let $z$ satisfy $\bar{A}_\bW z = 0$ as in~\eqref{eq:Aw-tight-z}.
The following are equivalent: \textup{($i$)}~$a_p \notin \operatorname{rowspan}(\bar{A}_\bW)$; \textup{($ii$)}~$z \ne 0$; \textup{($iii$)}~$a_p^\T z < 0$.
\end{lemma}

\proof
Clearly \textup{($iii$)}$\Rightarrow$\textup{($ii$)}.
Pre-multiplying~\eqref{eq:step_G} by $z^\T\bG$ and using $\bar{A}_\bW z = 0$, we get 
$a_p^\T z = -z^\T \bG z = -z^\T \bGs z \leq 0$, 
with equality iff $z = 0$, establishing \textup{($ii$)}$\Rightarrow$\textup{($iii$)}.
The implication \textup{($i$)}$\Rightarrow$\textup{($ii$)} holds as, 
if $z$ were $0$, then~\eqref{eq:step_G} would give
$a_p = \bar{A}_\bW^\T \bar{r}_\bW \in \operatorname{rowspan}(\bar{A}_\bW)$,
contradicting \textup{($i$)}.
To prove that \textup{($ii$)}$\Rightarrow$\textup{($i$)}, let $z\neq 0$
and assume by contradiction that
$a_p \in \operatorname{rowspan}(\bar{A}_\bW)$, i.e., that
$a_p = \bar{A}_\bW^\T v$; then, since $\bGs \succ 0$, we get 
$0>-z^\T \bGs z = a_p^\T z = v^\T \bar{A}_\bW z = 0$, which is not possible. In conclusion, we proved that \textup{($ii$)}$\Leftrightarrow$\textup{($iii$)}, \textup{($i$)}$\Leftrightarrow$\textup{($ii$)}, and, therefore, also that \textup{($i$)}$\Leftrightarrow$\textup{($iii$)}.
\cvd

The fact that Algorithm~\ref{alg:goldnash} maintains $\bar{A}_\bW$ full row rank is a property
inherited directly from the original GI algorithm, based on the following inductive argument
on the number of iterations.
At initialization $\bW=\emptyset$, so there is simply no linear dependence of rows in $\bar{A}_\bW$ as there are no rows. A constraint $p$ is added to $\bW$ (line~\ref{line:add}) only when $t_2\leq t_1$, which requires $t_2<+\infty$, hence $a_p^\T z<0$ (see lines~\ref{line:t2-if}--\ref{line:t2-end}); by Lemma~\ref{lem:descent} this implies
$a_p\notin\operatorname{rowspan}(\bar{A}_\bW)$, so the new row $a_p^\T$ is linearly independent
of all current rows of $\bar{A}_\bW$ and full row-rank is preserved. Removing a row from a full row rank matrix when dropping a constraint leaves a matrix that is still full row rank.

\begin{remark}
The original GI algorithm relies on $G=G^\T\succ 0$ to compute its Cholesky factorization $G=LL^\T$, and to use $L$ to compute quantities involving $G^{-1}$. In our case, we cannot compute a Cholesky factorization
due to the possible asymmetry of $\bG$, so we compute an LU factorization $P\bG=LU$ (Step~\ref{line:lu})
and use it to solve linear systems (Steps~\ref{line:lu2} and~\ref{line:init}) involving $\bG^{-1}$ via forward and backward substitution. The algorithm also requires the ability to solve the linear systems involving 
the possibly non-symmetric matrices $EY_E$ at initialization (Step~\ref{line:init}) and, during the iterations (Step~\ref{line:dual_step}), $\bar{S}_\bW$. For numerical efficiency, an LU decomposition of the latter can be maintained and updated incrementally during the iterations as the working set $\bW$ changes. 
\end{remark}

\section{Numerical Results}
\label{sec:results}
We compare the performance of Algorithm~\ref{alg:goldnash} (labeled as \texttt{g*nash}) against alternative methods for computing v-GNEs of strongly monotone LQ games: the proximal ADMM method of~\cite{BK21} (\texttt{admm}) extended to handle inequalities as described in~\cite[Section 4.2.1]{Bem26}; the log-barrier interior-point method of~\cite{LL25} (\texttt{ipm}); the DR-DAQP method of~\cite{ABB26} (\texttt{daqp}, new version 0.8.7); and Lemke's algorithm applied to the primal (\texttt{lemke}) and dual (\texttt{lemke$_{\texttt D}$}) LCP reformulations described in Section~\ref{sec:lcp_reformulation}. All these methods,
including a prototype Python implementation of Algorithm~\ref{alg:goldnash}, are implemented in or interfaced with the \texttt{NashOpt} library and called with default options. We do not include MILP-based methods to solve the KKT system~\eqref{eq:KKT} as described in~\cite{Bem26}, as they are mainly conceived to solve {\it non-variational} and possibly non-monotone LQ-GNEPs, nor KKT-residual methods, as they would be inefficient for the considered class of GNEPs.

\subsection{Randomly Generated LQ-GNEs}
Tables~\ref{tab:results}--\ref{tab:results_eq} report the mean CPU time over 100 random instances of v-GNE LQ games with up to $N=100$ agents, $n=5$ variables per agent, $m=2Nn$ coupling inequalities, and either no equalities or $q=\lfloor N/2 \rfloor$ equality constraints. Each problem instance is generated as follows: set tentative agents' QP Hessian matrices
$\tilde Q^{(i)} = B_i^\T B_i$, where each entry of $B_i$ is sampled i.i.d.\ from the standard normal distribution $\mathcal{N}(0,1)$, compute the resulting pseudogradient matrix $\tilde\bG$, and then set $Q^{(i)}=\tilde Q^{(i)} + (\max(-\lambda_{\rm min}(\tilde\bG),0)+10^{-4})I$ to ensure strong monotonicity.
Each entry of the linear terms $c^{(i)}$ is sampled from $\mathcal{N}(0,5^2)$, each upper (lower) bound is generated  uniformly from $[0.1,1]$ (resp.\ $[-1,-0.1]$). The coupling constraints are generated as follows: each entry of $A$ and $E$ is sampled from $\mathcal{N}(0,1)$, a feasible vector $x_0$ is sampled uniformly within the lower and upper bounds, then we set $f=Ex_0$ and $b=Ax_0+\tilde b$, where each entry of $\tilde b$ is sampled uniformly in $[0.1, 0.5]$.

The results show that Algorithm~\ref{alg:goldnash}, which successfully terminated in all problem instances, performs very well compared to the other methods, especially for large problem sizes with equality constraints\footnote{The results can be reproduced via \texttt{NashOpt}'s demo file \texttt{example\_lq\_vgne\_comparison.py}}. Primal Lemke is disabled for problems with $N>50$ agents due to the excessive CPU time required; \texttt{admm} and \texttt{ipm} are disabled for problems with $N>20$ agents for the same reason. The \texttt{daqp} method loses performance on average when the size increases, especially when equality constraints are present. Primal Lemke cannot be applied to equality-constrained problems. 

\begin{table}[ht]
\setlength{\tabcolsep}{4pt}
\renewcommand{\arraystretch}{1.}
\centering
\small
\caption{Mean CPU time (ms) over 100 random vGNE LQ instances (no equalities ($q=0$)) ($n=5$ variables per agent, $m=2Nn$ coupling inequalities).}
\begin{tabular}{r|cccccc}
\hline
$N$ & \texttt{lemke} & \texttt{lemke$_{\texttt D}$} & \texttt{g*nash} & \texttt{daqp} & \texttt{admm} & \texttt{ipm} \\
\hline
  2 & 0.05 & 0.04 & 0.06 & \textbf{0.04} & 16.22 & 9.07 \\
  3 & 0.09 & 0.04 & 0.05 & \textbf{0.04} & 21.88 & 9.39 \\
  5 & 0.25 & 0.08 & 0.07 & \textbf{0.05} & 34.46 & 11.18 \\
  10 & 1.61 & 0.35 & 0.16 & \textbf{0.11} & 115.56 & 42.81 \\
  20 & 10.87 & 1.18 & 0.46 & \textbf{0.21} & 305.43 & 128.22 \\
  30 & 33.67 & 2.11 & 0.85 & \textbf{0.47} & -- & -- \\
  50 & 146.68 & 7.08 & 3.40 & \textbf{1.71} & -- & -- \\
  100 & -- & 21.58 & 10.98 & \textbf{10.21} & -- & -- \\
\hline
\end{tabular}
\label{tab:results}
\end{table}

\begin{table}[ht]
\setlength{\tabcolsep}{4pt}
\renewcommand{\arraystretch}{1.}
\centering
\small
\caption{Mean CPU time (ms) over 100 random vGNE LQ instances (equalities ($q=\lfloor N/2 \rfloor$)) ($n=5$ variables per agent, $m=2Nn$ coupling inequalities).}
\begin{tabular}{r|cccccc}
\hline
$N$ & \texttt{lemke} & \texttt{lemke$_{\texttt D}$} & \texttt{g*nash} & \texttt{daqp} & \texttt{admm} & \texttt{ipm} \\
\hline
  2 & -- & 0.07 & 0.12 & \textbf{0.05} & 26.01 & 47.49 \\
  3 & -- & 0.14 & 0.12 & \textbf{0.06} & 36.67 & 14.25 \\
  5 & -- & 0.88 & 0.24 & \textbf{0.09} & 86.53 & 19.44 \\
  10 & -- & 1.89 & 1.25 & \textbf{0.32} & 448.04 & 336.67 \\
  20 & -- & 5.48 & 2.45 & \textbf{1.52} & 1688.05 & 606.78 \\
  30 & -- & 11.86 & \textbf{2.97} & 4.32 & -- & -- \\
  50 & -- & 35.69 & \textbf{9.38} & 18.28 & -- & -- \\
  100 & -- & 227.81 & \textbf{37.86} & 134.62 & -- & -- \\
\hline
\end{tabular}
\label{tab:results_eq}
\end{table}

\subsection{Game-Theoretic MPC}

We consider an example of game-theoretic linear MPC with $N=3$ agents operating on the
discrete-time LTI plant
\begin{equation}\label{eq:mpc_dyn}
  x(t{+}1) = Ax(t) + Bu(t), \quad y(t) = Cx(t)
\end{equation}
where $x(t)\in\R^{9}$ is the state vector, $u(t)\in\R^{6}$ is the input vector, and $y(t)\in\R^{6}$ is the output vector. Each agent $i$ controls two components of the input vector, denoted as $u_i(t)\in\R^2$, and aims to regulate the output $y$ to a common set-point $r$, each using different output weight matrices, while satisfying local input constraints and shared output constraints\footnote{This example can be reproduced via \texttt{NashOpt}'s demo file
\texttt{example\_mpc\_solver\_comparison.py}}.

We adopt the linear GT-MPC setup in~\cite{Bem26}, where each agent minimizes a finite-horizon cost over its own input trajectory, subject to local input constraints and shared soft output constraints.
By letting $\Delta u(t)=u(t)-u(t{-}1)$ denote the input increments, 
each agent~$i$ minimizes over
$(\Delta u_i(0),\ldots,\Delta u_i(T{-}1),\varepsilon_i)$ the finite-horizon cost
\begin{align}\label{eq:mpc_cost}
  J_i = \sum_{k=0}^{T-1}&\bigl[(y(k{+}1)-r(t))^\T Q_y^{(i)}(y(k{+}1)-r(t))\notag\\
  &{}+ \Delta u_i(k)^\T Q_{du}^{(i)}\,\Delta u_i(k)\bigr]
  + Q_\varepsilon\varepsilon_i + Q_{\varepsilon 2}\varepsilon_i^2
\end{align}
where $r(t)\in\R^{n_y}$ is the current output set-point, subject to the dynamics~\eqref{eq:mpc_dyn}, 
the shared soft output constraints
\begin{equation}\label{eq:mpc_shared}
  y_{\min} - \textstyle\sum_{j=1}^N\varepsilon_j\,\mathbf{1} \le y(k{+}1) \leq
  y_{\max} + \textstyle\sum_{j=1}^N\varepsilon_j\,\mathbf{1}
\end{equation}
where $\varepsilon_i$ is a per-agent scalar slack, and the local 
constraints
\begin{equation}\label{eq:mpc_local}
\begin{aligned}
  &u_{\min} \leq u_i(k) \leq u_{\max},&\varepsilon_i\geq 0\\
  &\Delta u_{\min} \leq \Delta u_i(k) \leq \Delta u_{\max}
\end{aligned}
\end{equation}
for $k=0,\ldots,T{-}1$.

We condense the problem by substituting the output trajectory $y(k{+}1)$ as a linear
function of the extended state $x_e(t)=[x(t)^\T,u(t{-}1)^\T]^\T$ and the stacked
input-increment sequence $\Delta u=(\Delta u(0),\ldots,\Delta u(T{-}1))$,
so that the overall GT-MPC problem takes the form~\eqref{eq:QPi} over
the joint decision vector
$z=(\Delta u_1(0{:}T{-}1)$, $\varepsilon_1$, $\ldots$, $\Delta u_N(0{:}T{-}1)$, $\varepsilon_N)$,
with no equality constraints.
The shared output constraints~\eqref{eq:mpc_shared} become the 
inequality constraints in~\eqref{eq:QPi} that couple the agents' decision variables.

The matrices $A$, $B$, $C$ are generated randomly with a block-diagonal dominant
structure (each agent owns a two-input/two-output subsystem with three states, with input $u_i$ and
output $y_i$) plus a small
inter-agent coupling on all matrix entries $\sim\mathcal{N}(0,0.02^2)$; after generating
the matrices, the spectral radius of $A$ is scaled
to~$0.95$ and $C$ is normalized so that the DC gain equals the identity.
We use diagonal weights $Q_y^{(i)}$ with diagonal entries equal to $1$ except for those corresponding to $y_i$,  which are set to $1.5$, $Q_{du}^{(i)}=0.1I$,
$Q_\varepsilon=10^3$, $Q_{\varepsilon 2}=10^{-3}$; each entry of
$u_{\min}=-3$, $u_{\max}=3$, $\Delta u_{\min}=-2$, $\Delta u_{\max}=2$,
$y_{\min}=0$, $y_{\max}=2$, and the set-point $r(t)\equiv[1\ 2\ 1\ 2\ 1\ 2]^\T$.
The resulting game is verified numerically to be strongly monotone for each value of~$T$.

We run a closed-loop simulation from $x_0=0$, solving
the variational GT-MPC problem at each step using
\texttt{g*nash}, \texttt{lemke}, \texttt{lemke$_{\texttt{D}}$}, and \texttt{daqp}
for prediction horizons $T\in\{10,15,20,25,30\}$.
Figure~\ref{fig:mpc_trajectories} shows the resulting closed-loop output and input trajectories for $T=30$.

Figure~\ref{fig:mpc_timing} reports the minimum, median, and maximum CPU time per
closed-loop step over the simulation.
Algorithm~\ref{alg:goldnash} scales better than Lemke and dual-Lemke method as $T$ grows, 
while DR-DAQP performs extremely well in all cases. 

\begin{figure}[ht]
\centering
\includegraphics[width=\columnwidth]{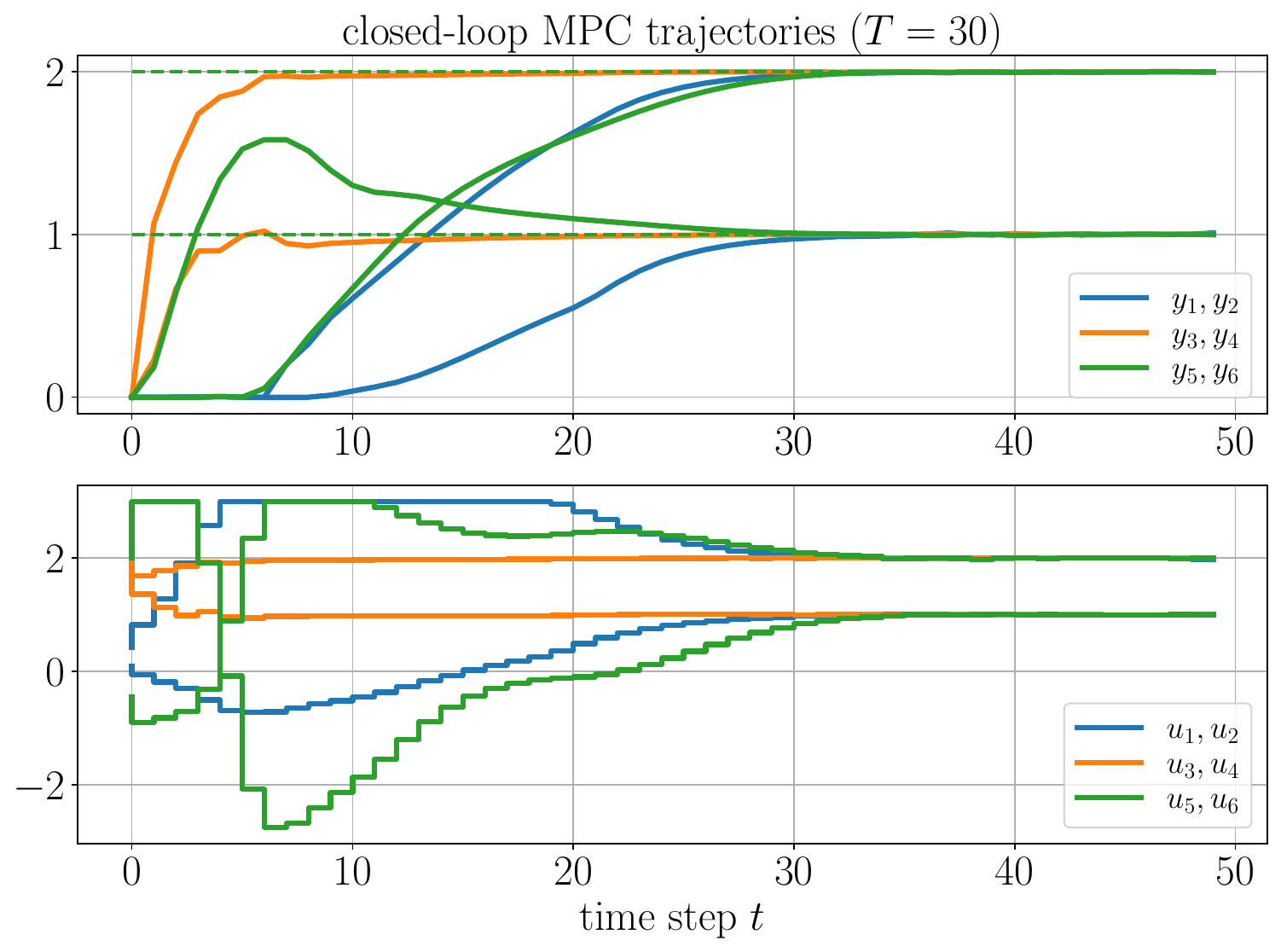}
\caption{Closed-loop output trajectories $y_1,\ldots,y_6$ (top, solid lines), output references (top, dashed lines), and input
trajectories $u_1,\ldots,u_6$ (bottom) for GT-MPC with prediction
horizon $T=30$. Colors indicate inputs (outputs) manipulated (more weighted) by each agent.}
\label{fig:mpc_trajectories}
\end{figure}

\begin{figure}[ht]
\centering
\includegraphics[width=\columnwidth]{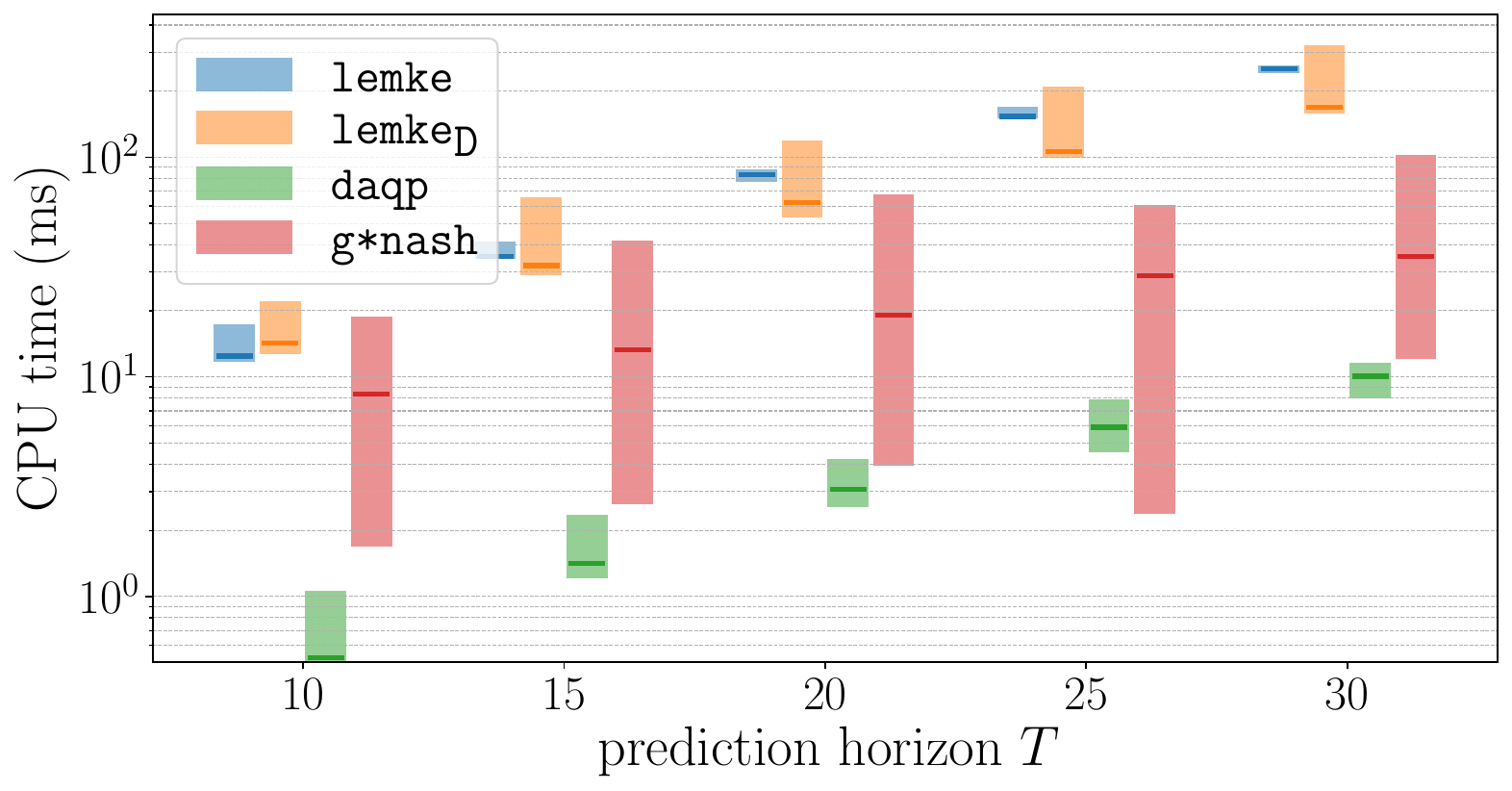}
\caption{CPU time per closed-loop step for GT-MPC ($N=3$ agents).
Bars show the min/max range, horizontal marks indicate the median.}
\label{fig:mpc_timing}
\end{figure}

\section{Conclusions}
\label{sec:conclusion}
We investigated a variant of the active-set method of Goldfarb-Idnani with the goal of computing variational
generalized Nash equilibria of strongly monotone $N$-player linear-quadratic games
with shared affine constraints. We verified that some key ingredients of the GI algorithm carry over from the symmetric QP case to the variational LQ-GNE setting without modification to the constraint add/drop logic.
Numerical experiments on random LQ games and a GT-MPC benchmark have shown
that the method can perform better than alternative solvers, especially as the problem size grows.
Further research is required to characterize conditions under which the method is guaranteed to converge to a v-GNE solution, when one exists, without the need to switch strategy upon detecting a cycle, and to possibly extend the method to non-variational equilibria.

\section*{Acknowledgments}
The author thanks Dr. Daniel Arnstr\"om for providing an updated version of the \texttt{DR-DAQP} solver and for  fruitful discussions.

\end{document}